\documentclass{article}

\RequirePackage[OT1]{fontenc}
\RequirePackage[
amsthm,amsmath
]{imsart}

\usepackage{graphicx}
\usepackage{latexsym,amsmath}
\usepackage{amsmath,amsthm,amscd}
\usepackage{amsfonts}
\usepackage[psamsfonts]{amssymb}
\usepackage{enumerate}

\usepackage{url}

\usepackage{tocvsec2}

\usepackage{epstopdf}
\usepackage{wrapfig}
\usepackage{float}

\startlocaldefs
\numberwithin{equation}{section}

\theoremstyle{plain} 
\newtheorem{theorem}{Theorem}[section]
\newtheorem{corollary}[theorem]{Corollary}

\newtheorem{lemma}[theorem]{Lemma}

\newtheorem{proposition}[theorem]{Proposition}

\theoremstyle{definition} 

\theoremstyle{definition} 

\newtheorem*{ex*}{Example}
\theoremstyle{remark} 

\theoremstyle{remark} 

\newtheorem*{remark*}{Remark}
\numberwithin{equation}{section}
 
\newcommand{\beqa}{\begin{eqnarray}}
\newcommand{\eeqa}{\end{eqnarray}}
 
\newcommand{\bseq}{\begin{subequations}}
 
\newcommand{\eseq}{\end{subequations}}

\newcommand{\BE}{{\operatorname{BE}}}

\newcommand{\al}{\alpha}
\newcommand{\g}{\gamma}

\newcommand{\si}{\sigma}

\newcommand{\ka}{\kappa}
\newcommand{\la}{\lambda}

\newcommand{\de}{\delta}
\newcommand{\be}{\beta}
\newcommand{\De}{\Delta}
\newcommand{\La}{\Lambda}
\newcommand{\vpi}{\varphi}
\renewcommand{\th}{\theta}
\renewcommand{\Psi}{\overline{\Phi}}

\newcommand{\ii}[1]{\,\mathbf{I}\{#1\}} 
\newcommand{\Bigii}[1]{\,\mathbf{I}\Big\{#1\Big\}} 

\newcommand{\pd}[2]{\frac{\partial#1}{\partial#2}}

\renewcommand{\P}{\operatorname{\mathsf{P}}} 
\newcommand{\E}{\operatorname{\mathsf{E}}}

\newcommand{\R}{\mathbb{R}}

\newcommand{\T}{\mathcal{T}}

\newcommand{\vp}{\varepsilon}

\newcommand{\tvp}{{\tilde{\vp}}}
\newcommand{\tbe}{{\tilde{\beta}}}

\newcommand{\ttau}{{\tilde{\tau}}}
\newcommand{\tth}{{\tilde{\theta}}}

\newcommand{\tD}{{\tilde{D}}}

\renewcommand{\le}{\leqslant}
\renewcommand{\ge}{\geqslant}

\endlocaldefs

\begin{document}

\begin{frontmatter}

\title{On the Berry--Esseen bound for the Student statistic}
\runtitle{Berry--Esseen for Student}

%

\begin{aug}
\author{\fnms{Iosif} \snm{Pinelis}\thanksref{t2}\ead[label=e1]{ipinelis@mtu.edu}}
  \thankstext{t2}{Supported by NSF grant DMS-0805946}
\runauthor{Iosif Pinelis}


\address{Department of Mathematical Sciences\\
Michigan Technological University\\
Houghton, Michigan 49931, USA\\
E-mail: \printead[ipinelis@mtu.edu]{e1}}
\end{aug}

\begin{abstract}
New Berry--Esseen-type bounds, with explicit constant factors, for the distribution of the Student statistic and, equivalently, for that of the self-normalized sum of independent zero-mean random variables are obtained. These bounds are compared with the corresponding existing results.  
\end{abstract}

  
%

\begin{keyword}[class=AMS]
\kwd[Primary ]{62E17}
\kwd{60E15}
\kwd[; secondary ]{62G10}
\kwd{62F03}
\end{keyword}


\begin{keyword}
\kwd{Berry--Esseen bounds}
\kwd{probability inequalities}
\kwd{independent random variables}
\kwd{Student statistic}
\kwd{self-normalized sum}
\end{keyword}

\end{frontmatter}

\settocdepth{chapter}

\tableofcontents 

\settocdepth{subsubsection}

\theoremstyle{plain} 
\numberwithin{equation}{section}

\eject


\section{Summary and discussion}\label{intro} 


Consider the self-normalized sum 
\begin{equation*}
	T:=\frac{S}{V},
\end{equation*}
where 
\begin{equation*}
	S:=\sum_1^n X_i,\quad	V:=\sqrt{\sum_1^n X_i^2}, 
\end{equation*}
and $X_1,\dots,X_n$ are independent zero-mean random variables (r.v.'s). 
It is assumed that $T=0$ on the event $\{V=0\}$. 
For any $p\in(0,\infty)$, introduce also 
\begin{gather*}
	\be_p:=\sum_1^n\E|X_i|^p\quad\text{and}\quad\tbe_p:=\sum_1^n\E|X_i^2-\E X_i^2|^{p/2}, 
\end{gather*}
assuming that $0<\be_3<\infty$ (and hence $0<\be_2<\infty$). 

Let $\Phi$ be the standard normal distribution function. 


\begin{theorem}\label{th:nonIID}
One has 
\begin{align}
	|\P(T\le z)-\Phi(z)|&\le A_3\frac{\be_3}{\be_2^{3/2}}+A_4\frac{\tbe_4^{1/2}}{\be_2}+A_6\frac{\tbe_6}{\be_3^3\be_2^{3/2}} 
	\label{eq:nonIID}
\end{align}
for all $z\in\R$ and 
for all triples $\tau:=(A_3,A_4,A_6)$ of absolute constants belonging to the set $\T:=\{\tau_1,\dots,\tau_4\}$ of triples, 
where 
\begin{align*}
\tau_1&:=(1.61, 1.60, 1.20), \\ 
\tau_2&:=(2.01, 1.02, 0.61), \\ 
\tau_3&:=(11.38, 11.02, 11.78\times10^{-6}), \\ 
\tau_4&:=(1.34, 125377, 1.049\times10^6). 
\end{align*}
\end{theorem}

The triple $\tau_1=(1.61, 1.60, 1.20)$ of the constant factors $A_3,A_4,A_6$ was obtained trying to minimize the maximum $A_3\vee A_4\vee A_6$ of the constants; for details, see the proof (in Section~\ref{proofs}) of Theorem~\ref{th:nonIID} and especially the table at the end of that proof. 
The triple $\tau_3$ was obtained trying to minimize the effect of the 6th-order moments of the $X_i$'s. 
The triple $\tau_4$ was designed to work best when $\tbe_4$ and $\tbe_6$ are very small, that is, when 
the distribution of each $X_i$ is close to the symmetric distribution on a symmetric two-point set. 
The triples $\tau_2,\tau_3,\tau_4$ will be used in this paper to compare the upper bound in \eqref{eq:nonIID} with one due to Shao \cite{shao05}. 

In the i.i.d.\ case, that is, when the r.v.'s $X_1,\dots,X_n$ are independent copies of a r.v.\ $X$, one can improve the values $A_3,A_4,A_6$ of the absolute constants in \eqref{eq:nonIID};  
at that, let us assume without loss of generality that 
$$\E X^2=1.$$ 
Introduce 
\begin{equation}\label{eq:rho's}
\rho_3:=\E|X|^3,\quad
\rho_4:=\sqrt{\E(X^2-1)^2},\quad
\rho_6:=\frac{\E|X^2-1|^3}{\E|X|^3}. 	
\end{equation}

\begin{theorem}\label{th:IID}
If $X,X_1,\dots,X_n$ are i.i.d.\ r.v.'s with $\E X=0$, $\E X^2=1$, and $\E|X|^3<\infty$, then 
\begin{align}
	|\P(T\le z)-\Phi(z)|&\le \frac{A_3\rho_3+A_4\rho_4+A_6\rho_6}{\sqrt n} 
	\label{eq:IID}
\end{align}
for all $z\in\R$ and 
for all triples $\tau:=(A_3,A_4,A_6)$ of absolute constants belonging to the set $\tilde\T:=\{\ttau_{1,1},\dots,\ttau_{4,1}\}$ of triples, 
where 
\begin{align*}
\ttau_{1,1}&:=(1.53, 1.52, 1.34), \\ 
\ttau_{1,2}&:=(1.61, 1.60, 1.02), \\ 
\ttau_{2,2}&:=(1.96, 1.02, 0.52), \\ 
\ttau_{2.1,1}&:=(1.96, 0.99, 0.63), \\ 
\ttau_{3,1}&:=(10.94, 9.40, 11.06\times10^{-6}), \\ 
\ttau_{4,1}&:=(1.25, 8140, 92437); 
\end{align*}
here, for each $i=1,2,3,4$, the triples $\ttau_{i,j}$ are to be compared with the triple $\tau_i$ in Theorem~\ref{th:nonIID}, with the same $i$. 
\end{theorem}

For $n\ge2$, the Student statistic 
\begin{equation*}
	t:=\frac{\overline{X}\sqrt n}{\sqrt{\frac1{n-1}\,\sum_1^n(X_i-\overline{X})^2}}, 
\end{equation*}
where $\overline{X}:=\frac1n\sum_1^n X_i$, 
can be expressed as a monotonic transformation of the self-normalized sum $T$: 
\begin{equation}\label{eq:t}
	t=\sqrt{\frac{n-1}n}\,\frac T{\sqrt{1-T^2/n}}. 
\end{equation}
Therefore, one immediately has 
\begin{corollary}\label{cor:}
  Theorems~\ref{th:nonIID} and \ref{th:IID} hold if $\P(T\le z)-\Phi(z)$ is replaced there by 
  $\P(t\le z)-\Phi_n(z)$, where 
\begin{equation}\label{eq:Phi_n}
	  \Phi_n(z):=\Phi\Big(\frac z{\sqrt{1+(z^2-1)/n}}\Big). 
\end{equation}
\end{corollary}

A Berry-Esseen type of bound of the optimal order for the Student statistic of i.i.d.\ $X_i$'s was obtained 
in 1996 by Bentkus and G{\"o}tze \cite{bent96}, using a Fourier transformation method. This was extended to the non-i.i.d.\ case by Bentkus, Bloznelis, and G{\"o}tze \cite{bbg96}, whose result can be rewritten as follows: 
\begin{equation}\label{eq:BG}
	|\P(t\le z)-\Phi\big(z\sqrt{\tfrac n{n-1}}\,\big)|
	\le C_2\g_2+C_3\g_3,
\end{equation}
where $C_2$ and $C_3$ are absolute constants, 
\begin{equation}\label{eq:ga's}
	\g_2:=\frac1{\be_2}\sum_1^n\E X_i^2\Bigii{|X_i|>\frac{\sqrt{\be_2}}2},\quad \g_3:=\frac1{\be_2^{3/2}}\sum_1^n\E|X_i|^3\Bigii{|X_i|\le\frac{\sqrt{\be_2}}2}. 
\end{equation} 

Note that $t\sim T$ as $n\to\infty$. The function $\Phi_n$, defined by \eqref{eq:Phi_n}, may be considered as an improper distribution function, with the ``impropriety'' $1-\big(\Phi_n(\infty)-\Phi_n(-\infty)\big)=2\big(1-\Phi(\sqrt n)\big)\sim\sqrt{\frac2{\pi n}}\,e^{-n/2}$ for large $n$, which is much less than $\frac1{\sqrt n}$. 
If $n$ is not very large, the tail probability $1-\Phi_n(z)$ may be much greater than $1-\Phi(z)$, which appears to correspond qualitatively to the fact that the tail of the Student distribution is significantly heavier than the standard normal tail when the number of degrees of freedom (d.f.) is not large. 
This heuristics appears to be confirmed by Figure~\ref{fig:tails}, for $n=10$; the pictures for $n=5$ and $n=20$ look quite similarly. 

\begin{figure}[htbp]
	\centering		
	\includegraphics[scale=.9]{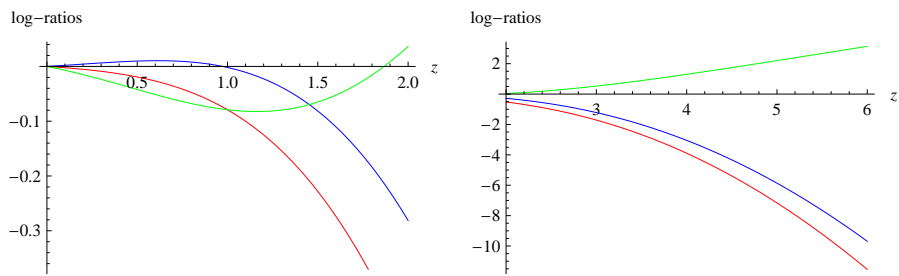}
	\caption{Logarithms of the ratios of the tail functions $1-\Phi(\cdot)$ (red), $1-\Phi\big(\cdot\sqrt{\tfrac n{n-1}}\big)$ (blue), and $1-\Phi_n(\cdot)$ (green) to the tail function of the Student distribution with $n-1$ d.f.}
	\label{fig:tails}
\end{figure}
\noindent It appears that on the interval $[1.5,\infty)$ the tail function $1-\Phi_n(\cdot)$ is closer to that of the Student distribution than the tail functions $1-\Phi(\cdot)$ and 
$1-\Phi\big(\cdot\sqrt{\tfrac n{n-1}}\big)$ are.
So, while the method of the proof (given in Section~\ref{proofs}) appears to allow one to obtain analogs of Theorems~\ref{th:nonIID} and \ref{th:IID} for $\P(t\le z)-\Phi(z)$ in place of $\P(T\le z)-\Phi(z)$ or $\P(t\le z)-\Phi_n(z)$, such analogs will not be pursued here.  


Anyway, the following proposition shows that $\Phi_n(z)$ differs from $\Phi(z)$ by much less than $1/\sqrt n$, uniformly in $z\in\R$. 

\begin{proposition}\label{prop:Phi_n-Phi}
For all $n>1$ and $z\in\R$
\begin{gather}
	|\Phi(z)-\Phi_n(z)|<\frac C{n-1}, \quad\text{where} \label{eq:Phi_n-Phi}\\ 
	C:=\Big(k-\frac12\Big)e^{-k}\sqrt{\frac k\pi}=0.162\dots\quad\text{and}\quad
	k:=1+\frac{\sqrt3}2;  \notag
\end{gather}
this constant factor, $C$, is the best possible in \eqref{eq:Phi_n-Phi}. 
\end{proposition}

One may be concerned that it is more natural to compare the distribution function of the statistic $t$ \big(as in \eqref{eq:t}, for general zero-mean $X_i$'s\big), not with $\Phi$ or $\Phi_n$, but with the distribution function (say $F_{n-1}$) of Student's distribution with $n-1$ d.f.\ --- that is, with the distribution function of the statistic $t$ for i.i.d.\ standard normal $X_i$'s. 
However, as shown in \cite{closeness-student}, 
\begin{equation*}
	|F_{n-1}(z)-\Phi(z)|<\frac{\tilde C}{n-1} \quad\text{with}\quad
	\tilde C=0.158\dots
\end{equation*}
for all $n\ge5$ and $z\in\R$. 
Therefore and in view of Proposition~\ref{prop:Phi_n-Phi}, 
$F_{n-1}(z)$ 
differs from $\Phi_n(z)$ by much less than $1/\sqrt n$, uniformly in $z\in\R$. 
Thus, Corollary~\ref{cor:} is quite relevant, notwithstanding the mentioned concern.

In the i.i.d.\ case, Nagaev \cite[(1.18)]{nag02} stated an inequality, which reads as follows (in the conditions of Theorem~\ref{th:IID}): for all $z\in\R$ 
\begin{equation}\label{eq:nag}
	|\P(T\le z)-\Phi(z)|<\Big(4.4\E|X|^3+\frac{\E X^4}{\E|X|^3}+\E|X^2-1|^3\Big)\,\frac1{\sqrt n}. 
\end{equation}
However, 
there are a number of mistakes 
in the proof of \eqref{eq:nag} in \cite{nag02}. 
It is also stated in \cite{nag02}, again in the i.i.d.\ case, that 
\begin{equation*}
	|\P(T\le z)-\Phi(z)|<\frac{36\E|X|^3+9}{\sqrt n}.  
\end{equation*}
Using Stein's method, Shao \cite{shao05} obtained a tighter and more general bound, also with explicit constants but without the i.i.d.\ assumption: 
\begin{align}
	|\P(T\le z)-\Phi(z)|&\le
	10.2\g_2+25\g_3 \label{eq:BS} \\
	&\le25{\be_p}/{\be_2^{p/2}} \notag
\end{align}
for all $p\in[2,3]$, with the same $\g_2$ and $\g_3$ as in \eqref{eq:ga's}. 
More recently, a Berry--Esseen bound for $T$ was obtained in \cite{bourguin} for i.i.d.\ standard normal $X_i$'s by means of Malliavin calculus. 

Let us compare the bounds in \eqref{eq:nonIID} and \eqref{eq:BS}. 
At that, let us restrict the attention to i.i.d.\ r.v.'s $X,X_1,\dots,X_n$. 

Consider first the case when $X$ has a two-point zero-mean distribution, so that \break 
$\P(X\in\{-a,b\})=1$ 
for some positive real numbers $a$ and $b$; that is,  
$$\P(X=b)=\frac a{a+b}=1-\P(X=-a).$$  
This case appears especially interesting, as any zero-mean distribution can be represented as a mixture of  two-point zero-mean distributions --- see e.g.\ \cite{disintegr}. 
Without loss of generality, assume that $b\ge a$ and $ab=1$. 
Then $b\ge1$ and $\E X^2=1$, 
and hence the bound in \eqref{eq:nonIID} (with the triple $\tau=\tau_3$ of constants $A_3,A_4,A_6$)  
is no greater than $(11.38\rho_3+11.02\rho_4+11.78\times10^{-6}\,\rho_6)/\sqrt n$, 
where again the $\rho_j$'s are as in \eqref{eq:rho's}, so that   
$\rho_3
=\frac{b^4+1}{b(b^2+1)}
$, 
$\rho_4
=b-1/b
$, and 
$\rho_6
=(b-1/b)^3
$.  
On the other hand, if $b>\sqrt{n}/2$, then 
the bound in \eqref{eq:BS} is no less than 
$10.2\frac b{b+1/b}\ge5.1
>1$. 
So, without loss of generality $b\le\sqrt{n}/2$ and hence 
the bound in \eqref{eq:BS} equals $25\rho_3/\sqrt{n}$. 
Thus (preferably with the help of the Mathematica command \verb9Reduce9 or similar tools), one finds that   
the bound in \eqref{eq:BS} will be less than the bound in \eqref{eq:nonIID} 
only if $b>469$, 
that is, only if the ``asymmetry index'' $b/a$ is greater than $469^2=219961$; 
at that, the inequality $b\le\sqrt{n}/2$ implies that $n$ must be no less that $(2b)^2>(2\times469)^2=879844$. 
One concludes that, for i.i.d.\ $X_i$'s with a common two-point distribution, \eqref{eq:nonIID} is better than \eqref{eq:BS} unless both the sample size $n$ and the asymmetry index are very large.  
Also, in the ``symmetric'' case when $b=a=1$, the bound in \eqref{eq:nonIID} (with $\tau=\tau_4$) reduces to $1.34/\sqrt n$, which is $\frac{25}{1.34}>18$ times as small as the bound in \eqref{eq:BS} \big(for $n\ge(2b)^2=4$\big). 

While the two-point distributions may be of particular interest, they are of a bounded support set, and hence all their moments are finite. On the other hand, one may object that the bounds given in Theorems~\ref{th:nonIID} and \ref{th:IID} will be infinite and hence useless if the 4th-order moments of the $X_i$'s are infinite. 
However, this concern is easily addressed via truncation. 

For a minute, let $X$ denote any zero-mean r.v. If the distribution of $X$ is continuous, then for each $b\in[0,\infty]$ there is some $a\in[0,\infty]$ such that the r.v.\ $X^{a,b}:=X\ii{-a<X<b}$ is zero-mean; the same holds in the case when the distribution of $X$ is symmetric (about $0$) --- then one can simply take $a=b$. If the zero-mean distribution of $X$ is not continuous or symmetric, one can use randomization, say as in \cite{disintegr}, to still find, for each $b\in[0,\infty]$, some $a\in[0,\infty]$ and some zero-mean r.v.\ $X^{a,b}$ such that $\P(-a\le X^{a,b}\le b)=1$ and $X^{a,b}=X$ on the event $\{-a<X<b\}$; let us refer to any such r.v.\ $X^{a,b}$ as a zero-mean truncation of the zero-mean r.v.\ $X$. 
\big(One could similarly base an appropriate construction on the so-called Winsorization $(-a)\vee(X\wedge b)$ instead of the truncation $X\ii{-a<X<b}$.\big)  

Now let $X_1,\dots,X_n$ be zero-mean r.v.'s as in Theorem~\ref{th:nonIID} or \ref{th:IID}. 
Respectively, let 
$B(X_1,\dots,X_n)$ denote (for any of the triples $\tau_1,\dots,\tau_4,\ttau_{1,1},\dots,\ttau_{4,1}$), either one of the bounds in \eqref{eq:nonIID} or \eqref{eq:IID}, as it depends on (the individual distributions of) the $X_i$'s. 
So, $B(X_1,\dots,X_n)$ denotes the bound in \eqref{eq:nonIID} under the conditions of Theorem~\ref{th:nonIID}, and it denotes the bound in \eqref{eq:IID} under the conditions of Theorem~\ref{th:IID}. 
The following corollary of Theorems~\ref{th:nonIID} and \ref{th:IID} is immediate: 

\begin{corollary}\label{cor:trunc}
Under the conditions of Theorem~\ref{th:nonIID} or \ref{th:IID},  
for each $i\in\{1,\dots,n\}$ let $X_i^{a_i,b_i}$ be a zero-mean truncation of $X_i$. 
Then for all $z\in\R$ 
\begin{equation}\label{eq:<B}
	|\P(T\le z)-\Phi(z)|\le\P\Big(\bigcup_1^n\{X_i\notin(-a_i,b_i)\}\Big)
	+B(X_1^{a_1,b_1},\dots,X_n^{a_n,b_n}). 
\end{equation}
\end{corollary}

Note that the upper bound in \eqref{eq:<B} can be expressed only in terms of the individual distributions of the $X_i$'s (rather than their joint distribution), since 
$$\P\Big(\bigcup_1^n\{X_i\notin(-a_i,b_i)\}\Big)=1-\prod_1^n\P\big(X_i\in(-a_i,b_i)\big).$$ 
So, when the bound in \eqref{eq:nonIID}, \eqref{eq:IID}, \eqref{eq:BG}, or \eqref{eq:BS} can be computed, usually the ``truncated'' bound in \eqref{eq:<B} can be computed as well. 

One may want to compare the bound in \eqref{eq:<B} with that in \eqref{eq:BS} or even with the ``truncated'' version of the latter bound: 
\begin{equation}\label{eq:BS-trunc}
	1-\prod_1^n\P\big(X_i\in(-a_i,b_i)\big)
	+10.2\tilde\g_2+25\tilde\g_3, 
\end{equation}
where $\tilde\g_2$ and $\tilde\g_3$ are obtained from $\g_2$ and $\g_3$ by replacing the $X_i$'s with their zero-mean truncations $X_i^{a_i,b_i}$, as in Corollary~\ref{cor:trunc}. 

Let us make such a comparison when the $X_i$'s are i.i.d.\ with a common distribution, which is either the  Student distribution with $d>0$ degrees of freedom or the (centered) Pareto distribution with the density 
\begin{equation*}
	f_s(x):=s\Big(x+\frac s{s-1}\Big)^{-s-1}\Bigii{x>-\frac1{s-1}}, 
\end{equation*}
where $s$ is a parameter with values in the interval $(1,\infty)$. 
Clearly, Student's distribution with $d$ degrees of freedom is symmetric, with heavy tails for small $d$ and light ones for large $d$, whereas the Pareto distribution with parameter $s$ is highly skewed to the right, with a heavy right tail for small $s>1$ and a light one for large $s$. 
In keeping with the ``i.i.d.'' assumption, let us consider the ``truncated'' bounds in \eqref{eq:<B} and \eqref{eq:BS-trunc} with $b_1=\dots=b_n=:b$ and, accordingly, $a_1=\dots=a_n=:a$; note that in each of the two cases under consideration (Student's or Pareto's), the value of $a$ is uniquely determined by that of $b$. 
Then,  
moreover, let us (numerically) minimize the ``truncated'' bounds in $b$. 
The results are shown in Figures~\ref{fig:t-distr5} and \ref{fig:pareto2}. 
There, the graphs are shown: of the bound in \eqref{eq:BS} (blue), of the minimized ``truncated'' bound \eqref{eq:BS-trunc} (magenta), of the bound in \eqref{eq:nonIID} (red), and of the minimized ``truncated'' bound in \eqref{eq:<B} (green) --- for sample sizes $n\in\{10,100,1000,10000\}$, $d\in[2.5,20]$, and $s\in[3.5,20]$; 
at that, for the ``red'' and ``green'' bounds the triple $\tau_2=(2.01, 1.02, 0.61)$ of constant factors in \eqref{eq:nonIID} is used. 

\begin{figure}[H] 
	\centering
		\includegraphics[width=.95\textwidth]{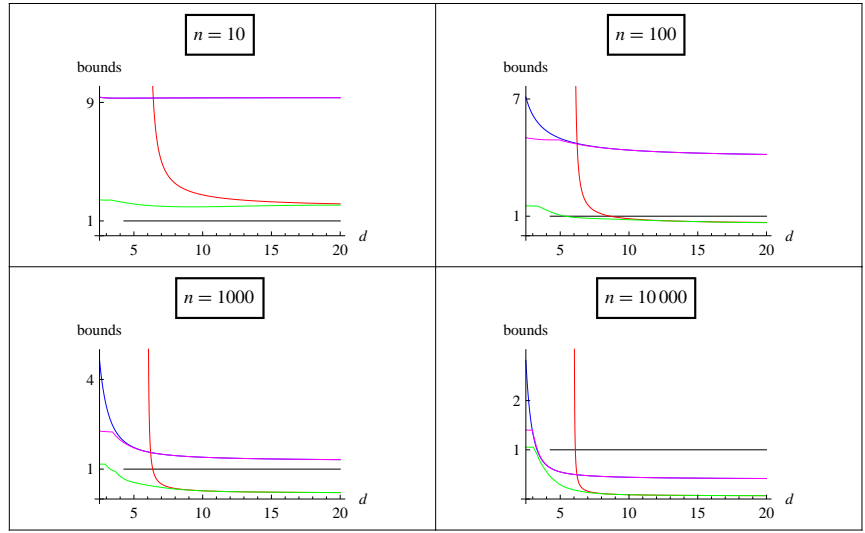}
	\caption{The bounds in the case of Student's distribution with $d$ degrees of freedom.}
	\label{fig:t-distr5}
\end{figure}

\begin{figure}[H] 
	\centering
		\includegraphics[width=.95\textwidth]{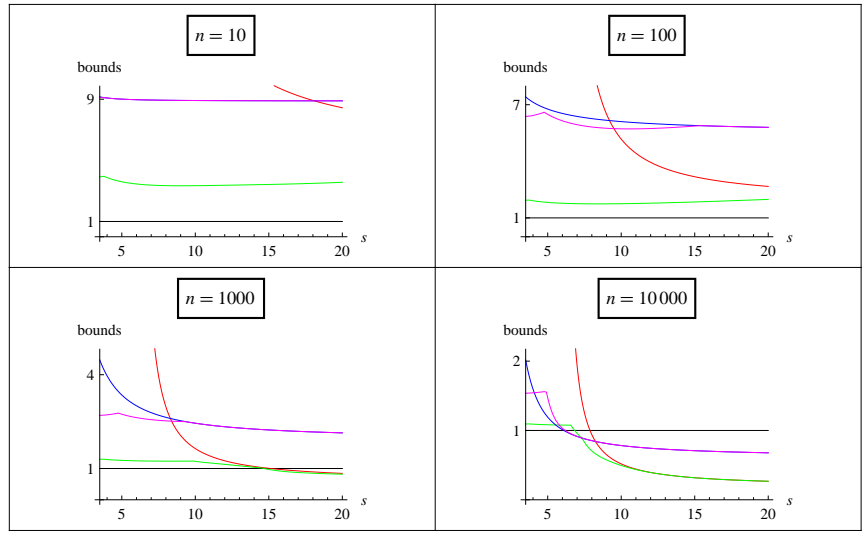}
	\caption{The bounds in the case of Pareto's distribution with parameter $s$.}
	\label{fig:pareto2}
\end{figure}

These pictures suggest the following. 
\begin{enumerate} 
	\item Predictably, truncation helps significantly 
	only when the tails are heavy enough --- that is, for small enough values of the parameters $d$ and $s$. Predictably as well, truncation is much more useful with the bound in 
\eqref{eq:nonIID} than it is with that in \eqref{eq:BS}. 
	\item For Student's and Pareto's distributions, even the minimized ``truncated'' bound in \eqref{eq:BS-trunc} is nontrivial (that is, less than $1$) only if $n$ is greater than $1000$ (or even a few thousands). 
	In fact, this bound is not much less than $0.5$ even for $n=10000$ and light tails.  
	For instance, for $n=10000$ and Student's distribution with $d=20$ d.f., the bound in \eqref{eq:BS} and the minimized bound in \eqref{eq:BS-trunc} are both $\approx0.417$, whereas the bound in \eqref{eq:nonIID} and the minimized bound in \eqref{eq:<B} are both $\approx0.068$ \big(again, with $\tau=\tau_2=(2.01, 1.02, 0.61)$\big). 
	\item Figure~\ref{fig:pareto2}, for the Pareto case, as as well as other considerations (see e.g.\ \cite{asymm,disintegr} and discussion therein) suggest that the Student statistic may not be appropriate for statistical inference when the underlying distribution is significantly skewed. 
Alternative statistics, ``correcting'' for the asymmetry, were offered and considered; see  \cite{asymm,disintegr} and discussion therein. 	 
	\item If the tails are very heavy, then even the minimized ``truncated'', ``green'' bound in \eqref{eq:<B} is not much less than $1$ even if $n$ is as large as $1000$ and the underlying distribution is symmetric. This may be in broadly considered agreement with the fact, established in \cite{shepp-etal_self-norm}, that if the the underlying distribution is in the domain of attraction of a stable law with index $\al<2$, then the limit distribution of the self-normalized sum and, equivalently, that of the Student statistic is not normal. 
	\item For almost all considered values of $n$, $d$, and $s$, the minimized ``truncated'' bound in \eqref{eq:<B} is significantly less than that in \eqref{eq:BS-trunc}, except in the Pareto case with $n=10000$ for a rather short interval of values $s$ near $7$, where, however, even the better bound is only slightly less than $1$. 
Conceivably, this deficiency might be fixed by using another triple of constants in place of the triple $\tau_2=(2.01, 1.02, 0.61)$. 
Moreover, when the tails are light enough, even the ``non-truncated'' bound in \eqref{eq:nonIID} significantly improves both on the ``truncated''  and ``non-truncated'' bounds in \eqref{eq:BS-trunc} and \eqref{eq:BS}.  
Thus, especially with the truncation tool, getting smaller constant factors may be more effective than insisting on the optimal order of moments even for the price of much greater constants. 
\end{enumerate}

It appears that, with the much smaller constant factors than in the preceding results, the bounds presented above may be approaching the state of being of use in statistical practice. 
There are additional resources to be tapped on. For instance, the proofs of Theorems~\ref{th:nonIID} and \ref{th:IID} rely to a large extent on a hybrid between the Chebyshev and Cantelli bounds, developed in \cite{between} specifically for the purposes of the present paper. One can similarly try to use and/or develop the much more accurate (but also much more complicated) upper bounds on large deviation probabilities given and discussed in \cite{pin-hoeff}; however, at that the proofs can be expected to be much harder to produce or read.  

\section{Proofs}\label{proofs}

\begin{proof}[Proof of Proposition~\ref{prop:Phi_n-Phi}]
Introduce $\La(a,z):=\Phi(u_{a,z})$, where $a\in(0,1)$ and $u_{a,z}:=\frac z{\sqrt{1+a(z^2-1)}}$, so that $\Phi_n(z)=\La(\frac1n,z)$ and $\Phi(z)=\La(0,z)$. 
By the mean value theorem, for some $b=b_z\in(0,a)$ 
\begin{equation*}
	\frac{\La(a,z)-\La(0,z)}a=\pd\La a(b,z)=\frac{u_{b,z}(1-u_{b,z}^2)\vpi(u_{b,z})}{2(1-b)},  
\end{equation*}
where $\vpi$ is the standard normal density function. 
So, to prove inequality \eqref{eq:Phi_n-Phi}, it suffices to note that $\sup_{u\in\R}|u(1-u^2)\vpi(u)|=2C$ and $\frac1{1-b}<\frac1{1-a}=\frac{n}{n-1}$ for $b\in(0,a)$ and $a=\frac1n$. 
That the  constant factor $C$ is the best possible in \eqref{eq:Phi_n-Phi} follows because, by l'Hospital's rule, $\frac{\La(a,z)-\La(0,z)}a\sim\pd\La a(a,z)$ as $a\downarrow0$. 
\end{proof}

The proof of Theorem~\ref{th:nonIID} is based, in part, on the following two lemmas. 

\begin{lemma}\label{lem:1}
Take any $\la$, $r_*$, $a$, $b$ in $(0,\infty)$. 
Take any $c$ and $r$ in $(0,\infty)$ such that 
\begin{equation*}
	c\ge\frac\la r\quad\text{and}\quad r\le r_*. 
\end{equation*}
Let $Y$ by any r.v.\ such that $\E Y=0$ and $\si:=\sqrt{\E Y^2}\in(0,\infty)$. Then 
\begin{equation}\label{eq:psi}
	\P(Y\ge c)\le\psi\Big(r_*,\frac\la\si\Big)r,\quad\text{where}\quad
	\psi(u,v):=\frac{u\wedge v}{v^2+(u\wedge v)^2}. 
\end{equation}
Also, 
\begin{equation}\label{eq:between}
	\P\big(Y\notin(-a,b)\big)\le\frac{4\si^2+(a-b)^2}{(a+b)^2}. 
\end{equation}
\end{lemma}

\begin{proof}[Proof of Lemma~\ref{lem:1}]
By the condition $c\ge\frac\la r$ and Cantelli's inequality, 
\begin{equation*}
	\P(Y\ge c)\le\P(Y\ge\tfrac\la r)\le\tfrac{\si^2}{\si^2+(\la/r)^2}
	=\tfrac{r^2}{r^2+v^2}, 
	\quad\text{where}\quad v:=\la/\si.   
\end{equation*}
Note that $\frac r{r^2+v^2}$ increases in $r\in[0,v]$ 
and decreases in $r\in[v,\infty)$. So, if $r_*\le v$, then the condition $r\le r_*$ implies 
$\frac r{r^2+v^2}\le\frac{r_*}{r_*^2+v^2}=\psi(r_*,v)$. 
If now $r_*\ge v$, then 
$\frac r{r^2+v^2}\le\frac v{v^2+v^2}=\psi(r_*,v)$, so that the inequality in \eqref{eq:psi} holds in this case as well. 
As for inequality \eqref{eq:between}, it is given in \cite{between}. 
\end{proof}

\begin{lemma}\label{lem:2}
For any positive real numbers $x,x_1,x_2$ such that $x\ge x_1\vee x_2$, 
one has 
\begin{equation}\label{eq:x Psi(x)}
	x_1\Psi(x)
	\le
	\Psi^*(x_2), 
\end{equation}
where 
\begin{align*}
\Psi&:=1-\Phi, \\
	\Psi^*(x)&:=0.17\ii{0<x<0.752}+x\Psi(x)\ii{x\ge0.752}. 
\end{align*}
\end{lemma}

\begin{proof}[Proof of Lemma~\ref{lem:2}]
It is well-known that the function $\Psi$ is log-concave; see e.g.\ \cite{HKP,pin99}. 
So, the function $L$ defined on $(0,\infty)$ by the formula $L(x):=\ln\big(x\Psi(x)\big)$ is concave, and hence $L(x)\le L(x_0)+L'(x_0)(x-x_0)$ for any $x$ and $x_0$ in $(0,\infty)$. 
Also, $L'(0.751)>0>L'(0.752)$ and $L(0.752)+L'(0.752)(0.751-0.752)<\ln0.17$. 
This implies that $L<\ln 0.17$ on $(0,\infty)$ and $L$ is decreasing on $(0.752,\infty)$, whence 
$\sup\{x\Psi(x)\colon x\ge z\}\le\Psi^*(z)$ for all $z\in(0,\infty)$. 
Now the lemma follows.  
\end{proof}

\begin{proof}[Proof of Theorem~\ref{th:nonIID}]  
This proof uses some of the ideas in the proof of \eqref{eq:nag} in \cite{nag02}, which were previously presented in \cite{novak00,novak05}. 
As mentioned before, there are a number of mistakes of various kinds 
in the proof in \cite{nag02}. 
For instance (in the notations of \cite{nag02}), a bound on $|\Phi(\frac{(1-\vp)r\si}{\sqrt{n}\,\si_n(r)})-\Phi(\frac{(1-\vp)r}{\sqrt{n}})|$ analogous to that on $|\Phi(\frac{r\si}{\sqrt{n}\,\si_n(r)})-\Phi(\frac{r}{\sqrt{n}})|$ in \cite[(1.12)]{nag02} is missing there; moreover, the same bound in \cite[(1.12)]{nag02} must have $(\frac\si{\si_n(r)}\wedge1)^2$ instead of $(\frac\si{\si_n(r)}\wedge1)$. 
We have also produced and utilized some new ideas in this proof. 
One of them is presented in Lemma~\ref{lem:1} above, which depends on the result of \cite{between}, specifically developed for the purposes of the present paper. 

Without loss of generality, assume that 
\begin{equation*}
	\be_2=1.
\end{equation*}
Take any 
\begin{equation}\label{eq:params}
\begin{gathered}
	\ka\in(0,\infty),\quad \vp_4\in(0,\tfrac12),\quad \vp_3\in(0,\infty),\quad \vp_2\in(0,1),\\  
	\th_3\in(0,1),\quad \th_4\in(0,\infty)  
\end{gathered}	
\end{equation}
and introduce  
\begin{gather}
\De:=\De(z):=\P(T\le z)-\Phi(z) \label{eq:De}\\
\intertext{and also}
r_3:=\be_3,\quad r_4:=\tbe_4^{1/2},\quad r_6:=\frac{\tbe_6}{\be_3^3}, \label{eq:r_j}\\ 
\vp:=\ka r_4,\quad \tvp_4:=\frac{\vp_4}\ka.  \label{eq:vp}
\end{gather}

It suffices to show that 
\begin{align*}
	|\De|=|\De(z)|&\le A_3 r_3+A_4 r_4+A_6 r_6, 
\end{align*}
where without loss of generality let us assume that  
\begin{equation*}
	z>0. 
\end{equation*}

Consider the following three cases.  

\emph{Case 1 (``small $n$''): $\vp\ge\vp_4$ or $r_3\ge\vp_3$.}\quad Note that 
\begin{equation*}
	\vp\ge\vp_4 \iff r_4\ge\tvp_4. 
\end{equation*}
So, 
\begin{align}
	|\De|\le1\le &
	(A_{3,1}\,r_3)\vee(A_{4,1}\,r_4)\vee(A_{6,1}\,r_6) \notag \\
	\le &A_{3,1}\,r_3+A_{4,1}\,r_4+A_{6,1}\,r_6,\quad\text{where} \label{eq:case1}\\
	&A_{3,1}:=\frac1{\vp_3},\quad A_{4,1}:=\frac1{\tvp_4},\quad A_{6,1}:=0.  \notag 
\end{align}


\emph{Case 2 (``large $n$'' \& ``large deviations''): $\vp<\vp_4$ \& $r_3<\vp_3$ \& 
$z\ge\frac{\th_3}{r_3}\wedge\frac{\th_4}{r_4}
$.}\quad 
Then, by \eqref{eq:De} and \eqref{eq:psi}, 
\begin{equation*}
	|\De|\le(P_1+P_2)\vee\Psi(z), 
\end{equation*}
where 
\begin{align*}
	P_1&:=\P\big(T>z,V>1-\vp_2\big) \\
	&\le\P(S>(1-\vp_2)z)
	\le\big(\psi(\vp_3,\tth_3)r_3)\vee\big(\psi(\tvp_4,\tth_4)r_4), \\
	\tth_j&:=(1-\vp_2)\th_j, \\
		P_2&:=\P\big(V\le1-\vp_2\big)
		=\P\Big(\sum_1^n(\E X_i^2-X_i^2)\ge\tvp_2\Big) 
		\le\psi(\tvp_4,\tvp_2)r_4, \\
	\tvp_2&:=\vp_2(2-\vp_2). 
\end{align*}
Note also that the currently assumed case conditions $\vp<\vp_4$ \& $r_3<\vp_3$ \& 
$z>\frac{\th_3}{r_3}\wedge\frac{\th_4}{r_4}$ imply 
$z>\frac{\th_3}{r_3}>\frac{\th_3}{\vp_3}$ or 
$z>\frac{\th_4}{r_4}>\frac{\th_4}{\tvp_4}$. 
So, Lemma~\ref{lem:2} 
yields 
\begin{equation*}
	\Psi(z)\le\big[\Psi^*(\tfrac{\th_3}{\vp_3})\,\tfrac{r_3}{\th_3}\big]
	\vee\big[\Psi^*(\tfrac{\th_4}{\tvp_4})\,\tfrac{r_4}{\th_4}\big]. 
\end{equation*}
Thus, 
\begin{align}
	|\De|
	\le &A_{3,2}\,r_3+A_{4,2}\,r_4+A_{6,2}\,r_6,\quad\text{where} \label{eq:case2} \\
	&A_{3,2}:=\psi(\vp_3,\tth_3)\vee\big[\Psi^*(\tfrac{\th_3}{\vp_3})\,\tfrac1{\th_3}\big], \notag \\ &A_{4,2}:=[\psi(\tvp_4,\tth_4)+\psi(\tvp_4,\tvp_2)]\vee\big[\Psi^*(\tfrac{\th_4}{\tvp_4})\,\tfrac1{\th_4}\big], \notag \\ 
	&A_{6,2}:=0. \notag  
\end{align}

\emph{Case 3 (``large $n$'' \& ``moderate deviations''): $\vp<\vp_4$ \& $r_3<\vp_3$ \& 
$z<\frac{\th_3}{r_3}\wedge\frac{\th_4}{r_4}$.}\quad In this case, note that 
\begin{equation*}
	\{T\le z\}=\{T_z\le z\},\quad\text{where}\quad
	T_z:=S-z\big(\sqrt{1+\eta}-1\big),\quad\text{and}\quad\eta:=V^2-1. 
\end{equation*}
Note also that the expression $S-z\big(\sqrt{1+\eta}-1\big)$ for $T_z$ is convex in $(S,\eta)$, so that 
its linear approximation \big(at the point $(\E S,\E\eta)=(0,0)$\big)
\begin{equation*}
	S_z:=S-z\eta/2
\end{equation*}
never exceeds $T_z$, whence 
\begin{equation*}
	\de:=\frac{T_z-S_z}z=1+\frac\eta2-\sqrt{1+\eta}\ge0. 
\end{equation*}
Therefore and because $\P(T\le z)=\P(T_z\le z)$, one has 
\begin{equation*}
	\P\big(S_z\le(1-\vp)z\big)-\P(\de>\vp)
	\le\P(T\le z)\le\P(S_z\le z). 
\end{equation*}
In view of \eqref{eq:De}, it follows that 
\begin{align*}
	\De&\le\BE+D(1)\quad\text{and} \\
	-\De&\le\BE+\P(\de>\vp)+D(1-\vp)+\tD_\vp,\quad\text{where} 
\end{align*}
\begin{align}
	\BE&:=\sup_{u\in\R}\Big|\P\big(S_z\le u\big)-\Phi\Big(\frac u{\si_z}\Big)\Big|, \notag \\
	\si_z&:=\sqrt{\E S_z^2}=\sqrt{\sum_1^n\E X_{i,z}^2}, \notag 
	\\
	D(u)&:=\Big|\Phi(uz)-\Phi\Big(\frac{uz}{\si_z}\Big)\Big|, \label{eq:D(u)}\\
	\tD_\vp&:=\sup_{x\ge0}\big[\Phi(x)-\Phi\big((1-\vp)x\big)\big]. \label{eq:tD_vp}
\end{align}
Thus, 
\begin{equation}\label{eq:De<,case3}
	|\De|\le\BE+\P(\de>\vp)+D(1)\vee D(1-\vp)+\tD_\vp. 
\end{equation}

Note also that 
\begin{gather*}
	S_z=\sum_1^n X_{i,z},\quad\text{where}\\
	X_{i,z}:=X_i-zY_i/2\quad\text{and}\quad Y_i:=X_i^2-\E X_i^2, \quad\text{whence}\\
	\eta=\sum_1^n Y_i. 
\end{gather*}
By a recent result of Shevtsova \cite{shev}, 
\begin{equation}\label{eq:BE<initial}
	\BE\le0.56\,\frac{\be_{3,z}}{\si_z^3},
\end{equation}
where
\begin{align}
	\be_{3,z}&:=\sum_1^n\E|X_{i,z}|^3
	\le\sum_1^n\E\Big(|X_i|+\frac z2|Y_i|\Big)^3
	\le\frac{\be_3}{(1-\al)^2}+\Big(\frac z2\Big)^3\frac{\tbe_6}{\al^2}, \label{eq:be3h} 
\end{align}
for any 
\begin{equation}\label{eq:al}
	\al\in(0,1); 
\end{equation}
the second inequality in \eqref{eq:be3h} follows from the elementary inequality 
$(a+b)^3\le\frac{a^3}{(1-\al)^2}+\frac{b^3}{\al^2}$ for all $a$ and $b$ in $[0,\infty)$ and $\al\in(0,1)$. 
Recalling also the condition 
$z<\frac{\th_3}{r_3}\wedge\frac{\th_4}{r_4}$ and definitions \eqref{eq:r_j}, one has  
\begin{equation}\label{eq:be3h over}
		\be_{3,z}
	\le\frac1{(1-\al)^2}\,r_3+\frac{\th_3^3}{8\al^2}\,r_6.  
\end{equation}
Next, 
\begin{equation}\label{eq:si_z=}
	\si_z^2=\sum_1^n\E X_{i,z}^2=1+\Big(\frac z2\Big)^2\tbe_4-z\sum_1^n\E X_i^3\ge1-z\be_3>1-\th_3. 
\end{equation}
So, \eqref{eq:BE<initial} and \eqref{eq:be3h over} yield 
\begin{equation}\label{eq:BE<final}
	\BE\le\frac{0.56}{(1-\th_3)^{3/2}}\Big(\frac1{(1-\al)^2}\,r_3+\frac{\th_3^3}{8\al^2}\,r_6\Big). 
\end{equation} 

Further, since 
$0<\vp<\vp_4<\frac12$, 
one has 
$\de>\vp\iff\eta\notin[2\vp-2\sqrt{2\vp},2\vp+2\sqrt{2\vp}\,]$. 
So, by \eqref{eq:between}, \eqref{eq:r_j}, and \eqref{eq:vp},  
\begin{equation}\label{eq:P(de<vp)<}
	\P(\de>\vp)\le\frac{4r_4^2+16\vp^2}{32\vp}=\frac{1+4\ka^2}{8\ka}\,r_4. 
\end{equation}

Next, by \eqref{eq:D(u)}, for any $u\in[0,\infty)$, 
\begin{equation}\label{eq:D(u)<}	
	D(u)\le uz\Big|\frac1{\si_z}-1\Big|
	\vpi\Big(\frac{uz}{\si_z\vee1}\Big), 
\end{equation}
where $\vpi$ is the standard normal density function. 
By the equalities in \eqref{eq:si_z=} and the case conditions $\vp<\vp_4$ and  
$z<\frac{\th_3}{r_3}\wedge\frac{\th_4}{r_4}$, 
\begin{align}
	|\si_z^2-1|&\le z\be_3+\Big(\frac z2\Big)^2\tbe_4
	=zr_3+\Big(\frac z2\Big)^2 r_4^2
	\le zr_3+\Big(\frac z2\Big)^2\tvp_4 r_4 \label{eq:|si_z-1|}\\
\intertext{and}	
	|\si_z^2-1|&\le zr_3+\Big(\frac z2\Big)^2 r_4^2\le\th_3+\th_4^2/4.  \label{eq:|si_z-1|,cont}
\end{align}
Writing $|\frac1{\si_z}-1|=\frac{|\si_z^2-1|}{\si_z+\si_z^2}$, and using \eqref{eq:D(u)<} and \eqref{eq:|si_z-1|}, one has 
\begin{gather*}
	D(u)\le D_1(u)+D_2(u),\quad\text{where} \\
	D_1(u):=r_3\,\frac{v^2\vpi(v)}u\,\rho_2,\quad
	D_2(u):=r_4\,\frac{\tvp_4}{4}\,\frac{v^3\vpi(v)}{u^2}\,\rho_3, \\
	v:=\frac{uz}{\si_z\vee1},\quad\rho_j:=\frac{(\si_z\vee1)^j}{\si_z+\si_z^2}. 
\end{gather*}
If $\si_z\le1$, then by \eqref{eq:si_z=} for $j=2,3$
\begin{equation*}
	\rho_j=\frac1{\si_z+\si_z^2}\le\rho_*:=\frac1{1-\th_3+\sqrt{1-\th_3}}. 
\end{equation*}
If $\si_z>1$, then by \eqref{eq:|si_z-1|,cont} for $j=2,3$ 
\begin{equation*}
	\rho_j=\frac{\si_z^j}{\si_z+\si_z^2}
	=\frac1{\si_z^{1-j}+\si_z^{2-j}}
	\le\rho_{**,j}:=\frac1{\si_*^{1-j}+\si_*^{2-j}},  
\end{equation*}
where
\begin{equation*}
	\si_*:=\sqrt{1+\th_3+\th_4^2/4}.
\end{equation*}
Note also that 
\begin{equation*}
\sup_{v>0}v^j\vpi(v)=s_j:=\frac1{\sqrt{2\pi}}\Big(\frac je\Big)^{j/2} 	
\end{equation*}
for $j=2,3$. 
Therefore, recalling also the condition $\vp<\vp_4$, one has 
\begin{equation}\label{eq:D<}
	D(1)\vee D(1-\vp)
	\le 
	r_3\,\frac{s_2}{1-\vp_4}\,(\rho_*\vee\rho_{**,2})
	+r_4\,\frac{\tvp_4}{4}\,\frac{s_3}{(1-\vp_4)^2}\,(\rho_*\vee\rho_{**,3}). 
\end{equation}

Next, let us estimate $\tD_\vp$. 
First here, one can use 
a special-case l'Hospital-type rule for monotonicity, such as \cite[Proposition~4.1]{pin06}, to see that for each $x\in(0,\infty)$ the ratio $\frac{\Phi(x)-\Phi((1-t)x)}t$ increases in $t\in(0,1)$. 
On the other hand, for each $t\in(0,1)$ the expression $\Phi(x)-\Phi\big((1-t)x\big)$ attains its maximum in  $x\in(0,\infty)$ at $x=x_t$, where 
\begin{equation*}
	x_t:=\sqrt{-\frac{2\ln(1-t)}{t(2-t)}}. 
\end{equation*}
On recalling also the definition \eqref{eq:tD_vp} of $\tD_\vp$ and the conditions $0<\vp<\vp_4<\frac12$, 
it follows that 
\begin{equation}\label{eq:tD<}
	\tD_\vp\le R(\vp_4)\vp=R(\vp_4)\ka r_4, \quad\text{where}\quad
	R(\vp_4):=\frac{\Phi(x_{\vp_4})-\Phi\big((1-\vp_4)x_{\vp_4}\big)}{\vp_4}. 
\end{equation}

Collecting \eqref{eq:De<,case3}, \eqref{eq:BE<final}, \eqref{eq:P(de<vp)<}, \eqref{eq:D<}, and \eqref{eq:tD<}, one bounds $|\De|$ in Case~3 as follows: 
\begin{align}
	|\De|
	\le &A_{3,3}\,r_3+A_{4,3}\,r_4+A_{6,3}\,r_6,\quad\text{where} \label{eq:case3}\\
	&A_{3,3}:=\frac{0.56}{(1-\th_3)^{3/2}(1-\al)^2}
	+\frac{s_2(\rho_*\vee\rho_{**,2})}{1-\vp_4},  \notag 
	\\ &A_{4,3}:=\frac{1+4\ka^2}{8\ka}
	+\frac{\tvp_4 s_3(\rho_*\vee\rho_{**,3})}{4(1-\vp_4)^2}
		+R(\vp_4)\ka, \notag \\ 
	&A_{6,3}:=\frac{0.07}{\al^2}\Big(\frac{\th_3^2}{1-\th_3}\Big)^{3/2}. \notag 
\end{align}

Collecting now the bounds \eqref{eq:case1}, \eqref{eq:case2}, and \eqref{eq:case3} on $|\De|$ in Cases~1--3, one concludes that in all of the three cases 
\begin{align}
	|\De|
	\le &A_3\,r_3+A_4\,r_4+A_6\,r_6,\quad\text{where} \label{eq:final}\\
	&A_p:=A_{p,1}\vee A_{p,2}\vee A_{p,3} \notag 
\end{align} 
for $p=3,4,6$. 

Now one can arbitrarily select positive ``weights'' $w_3,w_4,w_6$ and then 
try numerical minimization of (say) $(w_3A_3)\vee(w_4A_4)\vee(w_6A_6)$ with respect to all the parameters: $\al$, $\vp_4$, $\vp_3$, $\vp_2$, $\ka$, $\th_3$, $\th_4$, within their specified ranges --- recall \eqref{eq:params} and \eqref{eq:al}. 
The target function here appears to have a great number of local minima, and so, it is hardly possible to find the global minimum. 
Even though the numerical minimization is imperfect, it should be clear that the bound in \eqref{eq:final} holds for all the allowable values of the parameters as specified in \eqref{eq:params}. 
The following table shows the values of the parameters $\al$, $\vp_4$, $\vp_3$, $\vp_2$, $\ka$, $\th_3$, $\th_4$ found by the mentioned numerical minimization for each of a few selected triples $(w_3,w_4,w_6)$, as well as the resulting triple $\tau_i$ of the coefficients $(A_3,A_4,A_6)$, corresponding to the so obtained values of the parameters. 

\begin{center} 
\begin{tabular}{c|c|c||c|c|c|c|c|c|c||c}
$w_3$ & $w_4$ & $w_6$ & $\al$ & $\vp_4$ & $\vp_3$ & $\vp_2$ & $\ka$ & $\th_3$ & $\th_4$ & triple \\
	\hline	\hline
\rule[-6pt]{0pt}{18pt} $1$	& $1$	& $1$ &
$\frac{2}{25}$ & $\frac{123}{10^3}$ & $\frac{2703}{10^3}$ & $\frac{22}{125}$ & $\frac{43}{250}$ & $\frac{377}{10^3}$ & $\frac{5407}{10^3}$ & $\tau_1$		\\	
\hline					
\rule[-6pt]{0pt}{18pt} $1$	& $2$	& $1$ &
$\frac{27}{200}$ & $\frac{363}{10^3}$ & $\frac{1401}{10^3}$ & $\frac{19}{50}$ & $\frac{91}{250}$ & $\frac{413}{10^3}$ & $\frac{3167}{10^3}$ &$\tau_2$	\\	
\hline					
\rule[-6pt]{0pt}{18pt} $1$	& $1$	& $10^6$ &
$\frac{381}{500}$ & $\frac{471}{10^3}$ & $\frac{6927}{10^3}$ & $\frac{23}{10^3}$ & $\frac{79}{
   50}$ & $\frac{9}{200}$ & $\frac{3809}{10^3}$	& $\tau_3$	\\	
\hline					
\rule[-6pt]{0pt}{18pt} $1$	& $10^{-5}$	& $10^{-6}$ &
$\frac{8.39}{10^5}$ & $\frac{3.17}{10^5}$ & $1.32$ & $\frac{3.49}{10^5}$ & $\frac{9.97}{10^7}$ & $0.3738$ & $2.69$ & $\tau_4$	\\	
\hline					
\hline					
\end{tabular}	
\end{center}

Now Theorem~\ref{th:nonIID} is completely proved. 
\end{proof}

\begin{proof}[Proof of Theorem~\ref{th:IID}]  
This proof is quite similar to that of Theorem~\ref{th:nonIID}. 
The only essential difference that, instead of the constant $0.56$ in \eqref{eq:BE<initial} one can now use the better constant $0.4785$, according to a recent result of Tyurin \cite{tyurin}. 
Because we cannot find the global minima, it sometimes turns out that the numerical minimization with the better constant $0.4785$ produces results worse (or not quite better) than those obtained using the worse constant $0.56$. (!) In such cases, we used the values of the parameters $\al$, $\vp_4$, $\vp_3$, $\vp_2$, $\ka$, $\th_3$, $\th_4$ found in the general, non-iid setting ---   
with the worse constant $0.56$ and with the same weights $(w_3,w_4,w_6)$; the resulting triples are denoted as $\ttau_{i,2}$, with the second subscript $2$. Otherwise, the triple's second subscript is $1$, as in $\ttau_{1,1}$, $\ttau_{3,1}$, and $\ttau_{4,1}$. 
See the table below. 

\begin{center} 
\begin{tabular}{c|c|c||c|c|c|c|c|c|c||c}
$w_3$ & $w_4$ & $w_6$ & $\al$ & $\vp_4$ & $\vp_3$ & $\vp_2$ & $\ka$ & $\th_3$ & $\th_4$ & triple \\
	\hline	\hline
\rule[-6pt]{0pt}{18pt} $1$	& $1$	& $1$ &
$\frac{41}{500}$ & $\frac{113}{500}$ & $\frac{277}{100}$ & $\frac{39}{200}$ & $\frac{83}{500}$ & $\frac{409}{10^3}$ & $\frac{4467}{10^3}$	& $\ttau_{1,1}$ \\	
\hline					
\rule[-6pt]{0pt}{18pt} $1$	& $1$	& $1$ &
$\frac{2}{25}$ & $\frac{123}{10^3}$ & $\frac{2703}{10^3}$ & $\frac{22}{125}$ & $\frac{43}{250}$ & $\frac{377}{10^3}$ & $\frac{5407}{10^3}$ &$\ttau_{1,2}$	\\	
\hline						
\rule[-6pt]{0pt}{18pt} $1$	& $2$	& $1$ &
$\frac{27}{200}$ & $\frac{363}{10^3}$ & $\frac{1401}{10^3}$ & $\frac{19}{50}$ & $\frac{91}{250}$ & $\frac{413}{10^3}$ & $\frac{3167}{10^3}$	& $\ttau_{2,2}$ \\	
\hline
\rule[-6pt]{0pt}{18pt} $1$	& $2.1$	& $1$ &
$0.14$ & $\frac{275}{10^3}$ & $6.7$ & $0.42$ & $0.27$ & $0.44$ & $3.2$	& $\ttau_{2.1,1}$\\	
\hline						
\rule[-6pt]{0pt}{18pt} $1$	& $1$	& $10^6$ &
$\frac{777}{10^3}$ & $\frac{1}{2}$ & $\frac{1381}{500}$ & $\frac{27}{10^3}$ & $\frac{451}{10^3}$ & $\frac{47}{10^3}$ & $\frac{4569}{500}$ & $\ttau_{3,1}$	\\	
\hline					
\rule[-6pt]{0pt}{18pt} $1$	& $10^{-5}$	& $10^{-6}$ &
$\frac{3}{10^4}$ & $\frac{43}{10^5}$ & $10.3$ & $\frac{13}{10^4}$ & $3.5$ & $0.401$ & $1.6$	& $\tau_{4,1}$ \\	
\hline					
\hline					
\end{tabular}	
\end{center}
For instance, one can see that the values of the parameters $\al$, $\vp_4$, $\vp_3$, $\vp_2$, $\ka$, $\th_3$, $\th_4$ resulting in the triple $\ttau_{1,2}$ are the same those used to obtain the triple $\tau_1$. Similarly, the values of the parameters for the triple $\ttau_{2,2}$ are the same those for the triple $\tau_2$.  
\end{proof}

\bibliographystyle{abbrv}


\bibliography{C:/Users/Iosif/Documents/mtu_home01-30-10/bib_files/citations}

\end{document}